\documentclass[12pt]{article}
\usepackage{graphicx}
\usepackage{amsthm}
\usepackage{amssymb}
\usepackage{amsmath}
\graphicspath{{img/}}
\newtheorem{theorem}{Theorem}

\newtheorem{lemma}{Lemma}
\newtheorem{zmch*}{Remark}
\newtheorem{slds}{Corollary}
\newcommand{\eps}{\varepsilon}
\renewcommand{\l}{\lambda}
\renewcommand{\r}{\rho}
\newcommand{\Z}{\mathbb{Z}}
\newcommand{\R}{\mathbb{R}}

\title{A formula for the HOMFLY polynomial of rational links}
\author{S.\,Duzhin\thanks{Supported by grants RFBR 08-01-00379 and
NSh-8462.2010.1.}, M.\,Shkolnikov}
\date{}

\begin{document}
\maketitle
\begin{abstract}
In this paper we give an explicit formula for the HOMFLY polynomial of a
rational link (in particular, knot) in terms of a special continued
fraction for the rational number that defines the given link.
\end{abstract}

\section{Rational links}
\label{rat_links}

Rational (or 2-bridge) knots and links constitute an important class of links
for which many problems of knot theory can be completely solved and
provide examples often leading to general theorems about
arbitrary knots and links.
For the basics on rational (2-bridge) knots and links we refer the reader
to \cite{Lik} and \cite{KM}. As regards the definition, we follow \cite{Lik},
while the majority of properties that we need, are to be found in a more
detailed exposition of \cite{KM}. In particular, by \textit{equivalence}
of (oriented) links $L=K_1\cup K_2$ and $L'=K_1'\cup K_2'$
we understand a smooth isotopy of $\R^3$ which takes the union $K_1\cup K_2$
into the union $K_1\cup K_2$, possibly interchanging the components of the
link.\footnote{In fact \cite{HM}, all rational links are 
\textit{interchangeable}, that is, there is an isotopy of a rational link
onto itself that interchanges the components.}

Let $p$ and $q$ be mutually prime integers, $p>0$, $|{p\over q}|\leq 1$, 
and we have a continued fraction
\begin{equation}
{p \over q}=
\cfrac{1}{
b_1+\cfrac{1}{
b_2+\cfrac{1}{
\dots+\cfrac{1}{
b_{n-1}+\cfrac{1}{b_n}}}}}\ ,
\end{equation} 
where $b_i$ are nonzero integers (positive or negative).
Below, we will use shorthand notation\label{short_cf} $[b_1,b_2,\dots,b_n]$
for the continued fraction with denominators $b_1,b_2,\dots,b_n$.
A theorem of Schubert (see, for instance, \cite{Lik,KM}) says that the 
(isotopy type of the) resulting unoriented link does not depend on the choice 
of the continued fraction for the given number $p/q$.

The case $p=q=1$ is exceptional: it corresponds to the
trivial knot which is the only rational, but not 2-bridge knot.
On some occasions, it will be helpful to allow the numbers $b_i$ also take
values $0$ and $\infty$ subject to the rules $1/0=\infty$, $1/\infty=0$,
$\infty+x=\infty$. 

Consider a braid on four strands corresponding to the word
$A^{b_1}B^{b_2}A^{b_3}\dots$, where $A$ and $B$ are fragments
depicted in Figure \ref{elements} and concatenated from left to right.

\begin{figure}[h]
$$\begin{array}{ccccccc}
\includegraphics[height=15mm]{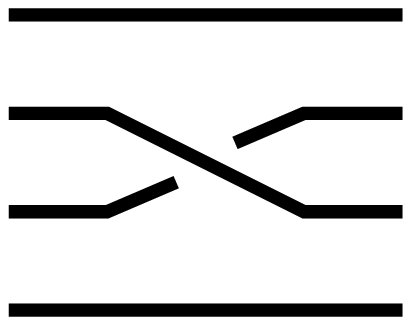} 
&\qquad&
\includegraphics[height=15mm]{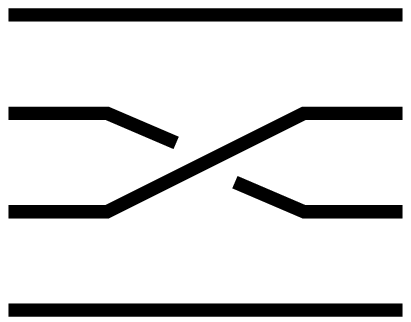} 
&\qquad&
\includegraphics[height=15mm]{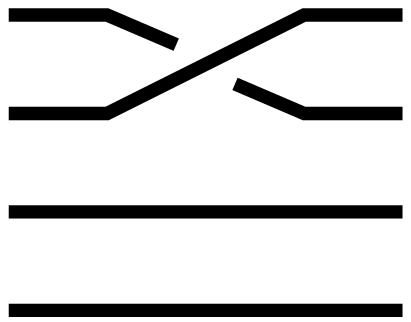} 
&\qquad&
\includegraphics[height=15mm]{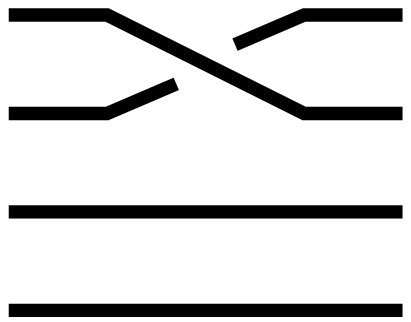} \\
A &\qquad& A^{-1} &\qquad& B &\qquad& B^{-1}
\end{array}$$
\caption{Fragments of natural diagrams}
\label{elements}
\end{figure}

Then take the closure of this braid depending on the parity
of $n$ (see Fig. \ref{closure}).

\begin{figure}[h]
\begin{center}
\begin{minipage}[h]{0.49\linewidth}
\center{\includegraphics[width=0.5\linewidth]{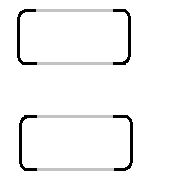} \\ $n\equiv 1\mod 2$}
\end{minipage}
\begin{minipage}[h]{0.49\linewidth}
\center{\includegraphics[width=0.5\linewidth]{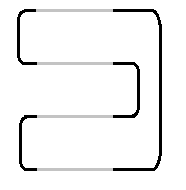}  \\ $n\equiv 0\mod 2$}
\end{minipage}
\caption{Odd and even closure}
\label{closure}
\end{center}
\end{figure}

We will call (non-oriented) diagrams obtained in this way 
\textit{natural diagrams of rational links} and denote them by 
$D[b_1,b_2,\dots,b_n]$.
We shall denote the link represented by this diagram as 
$L({p\over q})$.  For odd denominators $L({p\over q})$
turns out to be a knot, while for even denominators it is a two-component
link. Such knots and links are called \textit{2-bridge} or
\textit{rational}.
\medskip

\textbf{Example.}
We have, among others, the following two continued fractions for the rational
number 4/7 (we use shorthand notation, see page \pageref{short_cf}):
$$
  {4\over 7} = [1,1,3] = [2,-4].
$$
These fractions correspond to the natural link diagrams shown in Figure
\ref{nat_diag}.

\begin{figure}[ht]
\begin{center}
\includegraphics[height=14mm]{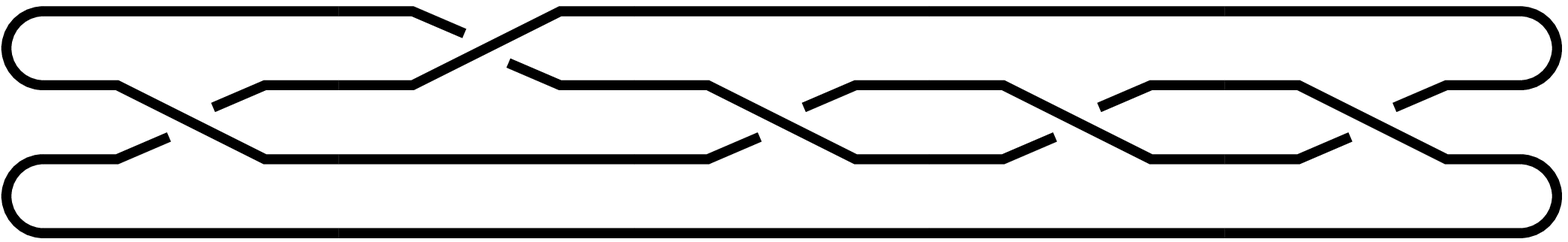}\vspace{3mm}
\includegraphics[height=14mm]{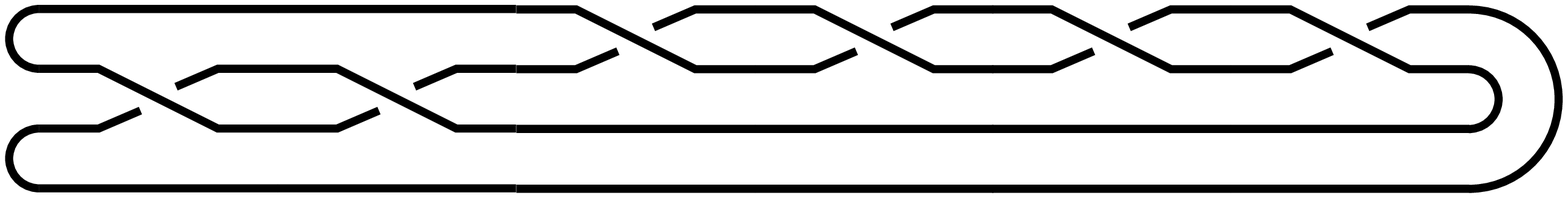}
\caption{Two natural diagrams of the table knot $5_2$}
\label{nat_diag}
\end{center}
\end{figure}

\section{Orientations}

Note that, if a natural diagram represents a two-component link, then the
two vertical leftmost fragments belong to different components.  If they are
oriented in the same direction, as shown in Figure \ref{pos_orient}, then we
call the diagram \textit{positive} and denote it by
$D^+[b_1,b_2,\dots,b_n]$.  \begin{figure}[h]
\begin{center}
\begin{minipage}[h]{0.49\linewidth}
\center{\includegraphics[width=0.5\linewidth]{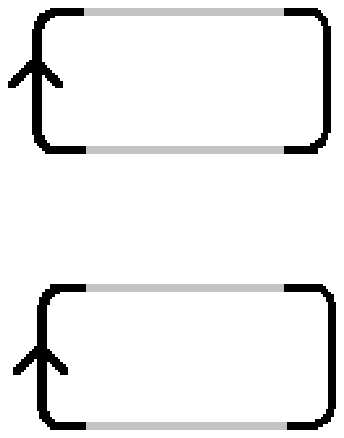} \\ $n\equiv1 \mod  2)$}
\end{minipage}
\begin{minipage}[h]{0.49\linewidth}
\center{\includegraphics[width=0.5\linewidth]{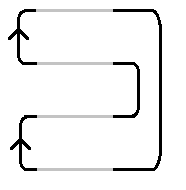}  \\ $n\equiv0 \mod  2)$}
\end{minipage}
\caption{Positive orientation on a 2-component rational link}
\label{pos_orient}
\end{center}
\end{figure}

If the orientation of one of the components is
reversed, then we call it \textit{negative} and denote by
$D^-[b_1,b_2,\dots,b_n]$. 
It does not matter which component of the link is
reversed, because the change of orientation of both components yields the
same link, see \cite{KM}.
As we will see later (Lemmas \ref{plusminus} and \ref{dual}), the
choice between the corresponding links does not depend on a particular
continued fraction expansion of the number $p/q$.  This makes the notations
$L^+(p/q)$ and $L^-(p/q)$ well-defined.

Let $p'=p-q$, if $p>0$, and $p'=p+q$, if $p<0$. According to \cite{KM}, we
have: $L^-({p\over q})=L^+({p'\over q})$, therefore, in principle, it is
sufficient to study only the totality of all positive rational links.  In
the case of knots (when $q$ is odd), the two oppositely oriented knots are
isotopic, and we have $L({p\over q})=L({p'\over q})$\label{chnge_knot}
(again, see \cite{KM}).  Therefore, it is sufficient to study only the
knots with an even numerator (cf.  Lemma \ref{even_prod} below).

Another important operation on links is the reflection in space;
it corresponds to the change of sign of the corresponding rational number:
$p/q\mapsto-p/q$, see \cite{KM}. 

The two symmetry operations on rational links generate a group
$\Z_2\times\Z_2$; they are transparently exemplified
by the examples $p/q=1/4, -1/4, 3/4, -3/4$, which correspond to the four
versions of the so called \textit{Solomon knot} (although it is actually a
two-component link):

\begin{align*}
L^+(1/4)\leftrightarrow D^+[4]&=\raisebox{-7mm}{\includegraphics[height=14mm]{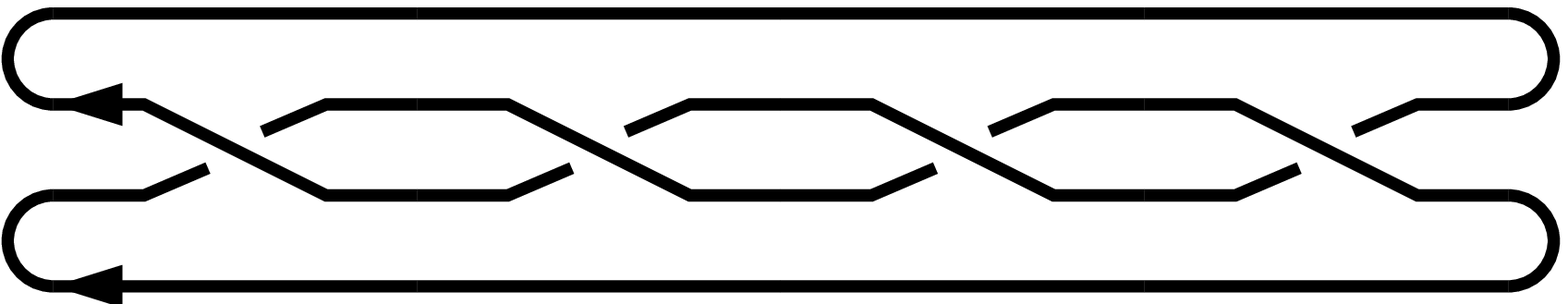}}\\[2mm]
L(-1/4)\leftrightarrow D^+[-4]&=\raisebox{-7mm}{\includegraphics[height=14mm]{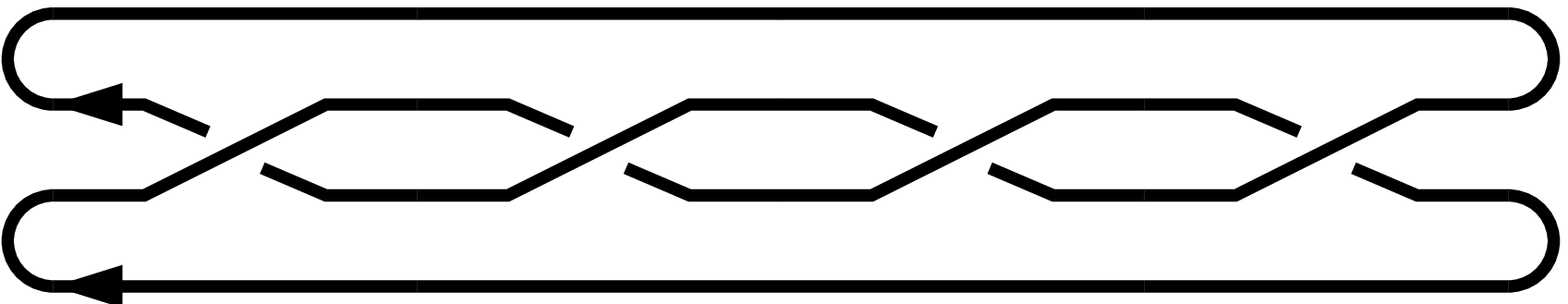}}\\[2mm]
L(3/4)\leftrightarrow D^+[1,3]&=\raisebox{-7mm}{\includegraphics[height=14mm]{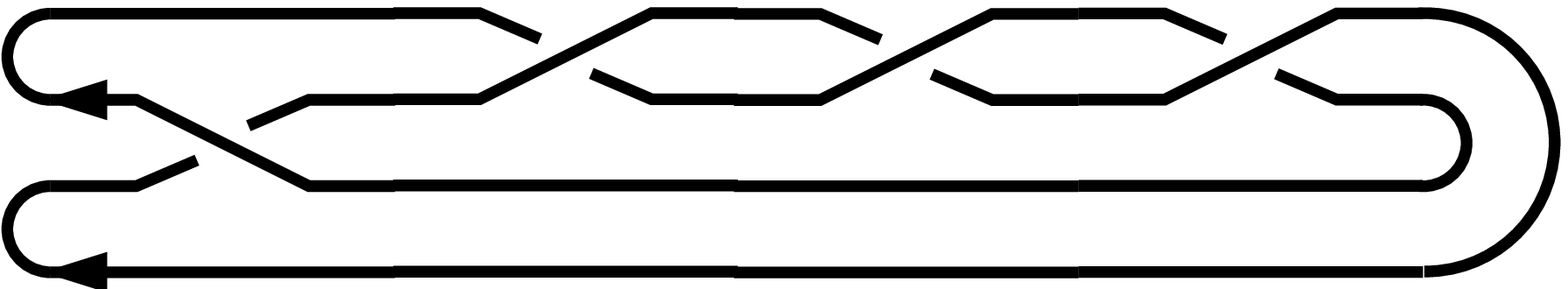}}\\[2mm]
L(-3/4)\leftrightarrow D^+[-1,-3]&=\raisebox{-7mm}{\includegraphics[height=14mm]{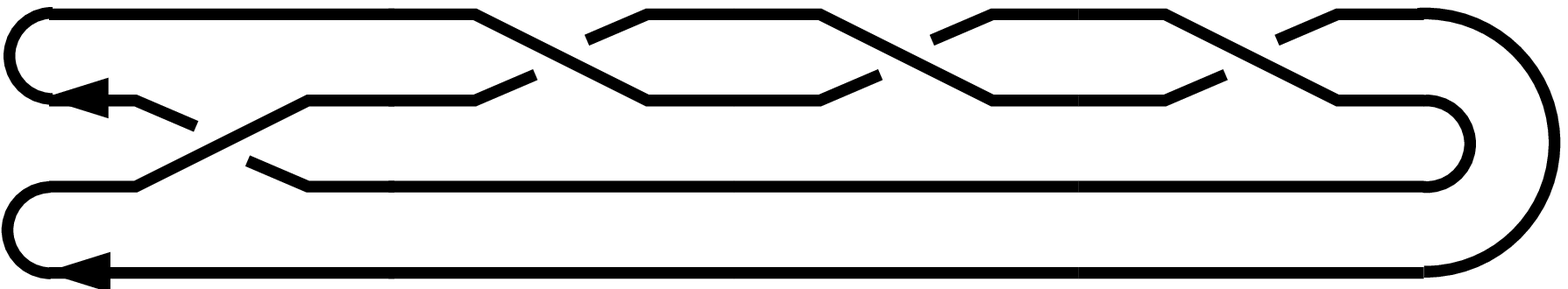}}
\end{align*}

By dragging the lower strand of the diagram for $L(3/4)$ upwards we get 
the diagram for $L(-1/4)$ with the opposite orientation of the upper strand.
The same is true for the pair $L(-3/4)$ and $L(1/4)$.

\section{HOMFLY polynomial}
\label{HOMFLY}

In 2004--2005 Japanese mathematicians S.\,Fukuhara \cite{Fuk}
and Y.\,Mizuma \cite{Miz} found
independently different explicit formulae for the simplest invariant
polynomial of 2-bridge links: the Conway (Alexander) polynomial.
The aim of the present paper is to establish a formula for a more general
HOMFLY polynomial $P$ in terms of the number $p/q$ that defines the
rational link.

The HOMFLY polynomial \cite{Lik, PS, CDM} is a Laurent polynomial in two
variables $a$ and $z$ uniquely defined by the following relations (we use
the normalization of \cite{Atlas} and \cite{CDM}; other authors may write 
the same polynomial in different pairs of variables, for example, Lickorish
\cite{Lik} uses $l=\sqrt{-1}a$ and $m=-\sqrt{-1}z$):

\begin{equation}
\label{skein}
P(\bigcirc)=1,\ aP(L_+)-a^{-1}P(L_-)=zP(L_0),
\end{equation}
where $L_+$, $L_-$ and $L_0$ are links that differ inside a certain ball
as shown in Figure \ref{homfly_fig}.

\begin{figure}[h]
\begin{center}
\begin{minipage}[h]{0.25\linewidth}
\center{\includegraphics[width=15mm]{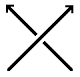} \\ $L_+$}
\end{minipage} \qquad
\begin{minipage}[h]{0.25\linewidth}
\center{\includegraphics[width=15mm]{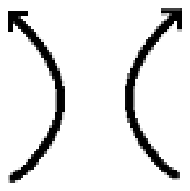} \\ $L_0$}
\end{minipage} \qquad
\begin{minipage}[h]{0.25\linewidth}
\center{\includegraphics[width=15mm]{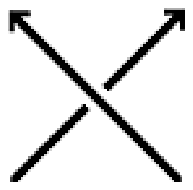} \\ $L_-$}
\end{minipage}
\caption{Outside of these regions the three links coincide}
\label{homfly_fig}
\end{center}
\end{figure}

As we mentioned in the previous section, in the case of rational knots, the
change of orientation gives the same (isotopic) knot, while for links it is
important to distinguish between the two essentially different orientations
(this number is two, not four, because the change of orientation on both
components gives the same rational link).

There is a simple formula relating the HOMFLY polynomials of a knot (link)
with that of its mirror reflection ($a\mapsto-a^{-1}$, $z\mapsto z$),
so in principle it is enough to study only the knots (links) described by
positive fractions.

HOMFLY polynomials of some links are given below in Figure
\ref{homfly_torus} and Table \ref{homflytab}.

\section{Reduction formula}
\label{reduc}

Consider a family of links $L_n$ for $n$ even, 
which coincide everywhere but in a certain ball, where they
look as shown in Fig. \ref{counter}.a, \ref{counter}.b and \ref{counter}.c. 
Moreover, we define the link $L_\infty$ by Fig. \ref{counter}.d.
That is, we consider a family of links with a distinguished block
where the strands are counter-directed. A formula similar to what we are
going to prove, can also be established for co-directed strands, but for our
purposes the following Proposition is sufficient. It expresses the value 
$P(L_n)$ through $P(L_0)$ and $P(L_\infty)$.

\begin{figure}[h]
\begin{minipage}[h]{0.45 \linewidth}
\center{\includegraphics[width=1\linewidth]{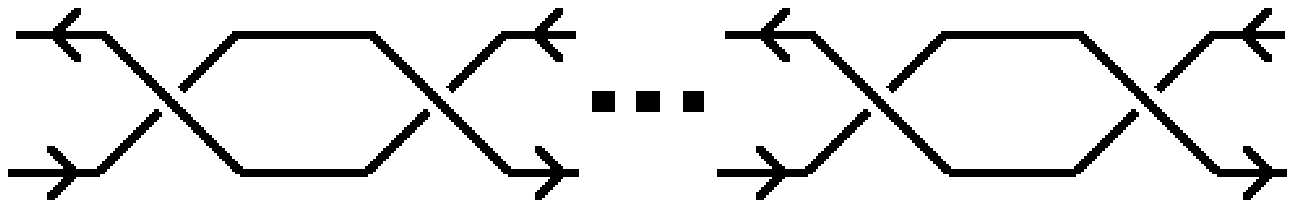}} a)\ $n>0$ \\
\end{minipage}
\hfill
\begin{minipage}[h]{0.45\linewidth}
\center{\includegraphics[width=1\linewidth]{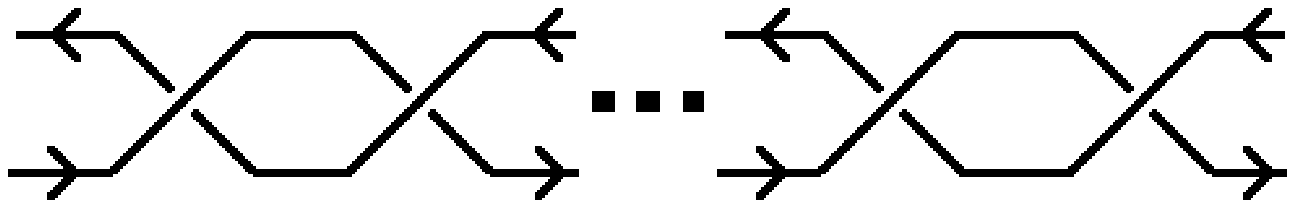}} b)\ $n<0$ \\
\end{minipage}
\vfill
\begin{minipage}[h]{0.45\linewidth}
\center{\includegraphics[width=1\linewidth]{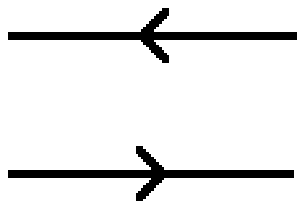}} c)\ $n=0$ \\
\end{minipage}
\hfill
\begin{minipage}[h]{0.45\linewidth}
\center{\includegraphics[width=1\linewidth]{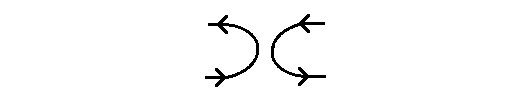}} d)\ $n=\infty$ \\
\end{minipage}
\caption{The differing portions of the links $L_n$.
In the first two pictures the elementary fragment is repeated
$|n|\over 2$ times.}\label{even_fragm}
\label{counter}
\end{figure}

\begin{lemma} \label{P_of_Ln}
$$
  P(L_n)=a^nP(L_0)+z{{1-a^n}\over{a-a^{-1}}}P(L_\infty)\ .
$$
\end{lemma}

\begin{proof}
Proceed by induction on $n$.

1) For $n=0$ the assertion is trivially true.

2) Suppose it is true for $n-2$.
The skein relation (\ref{skein}) shows that $zP(L_\infty)=aP(L_{n-2})-a^{-1}P(L_n)$.
Substituting here the assumed formula for $P(L_{n-2})$,
we can express $P(L_n)$ as follows:
\begin{align*}
P(L_n)&=a^2P(L_{n-2})-zaP(L_\infty)\\
&=a^2\big(a^{n-2}P(L_0)+z{1-a^{n-2}\over a-a^{-1}}P(L_\infty)\big)-zaP(L_\infty)\\
&=a^nP(L_0)+z\big(a^2{1-a^{n-2}\over a-a^{-1}}-a\big)P(L_\infty)\\
&=a^nP(L_0)+z{1-a^n\over a-a^{-1}}P(L_\infty)\ .
\end{align*}
The positive branch of induction is thus proved.

3) Suppose the assertion holds for a certain value of $n$. Prove it for
the value $n-2$. To do so, it is enough to reverse the argument in the
previous item. This completes the proof of the proposition.
\end{proof}

\begin{zmch*} For even values of $n$ the fraction $(1-a^n)/(a-a^{-1})$
is actually a Laurent polynomial, namely, $-a-a^3-\dots-a^{n-1}$, if $n>0$,
and $a^{-1}+a^{-3}+\dots+a^{n+1}$, if $n<0$.
\end{zmch*}

\noindent
\begin{minipage}[h]{0.7\linewidth}
\begin{slds}
Let $T_{2,n}$ be the torus link with counter-directed strands 
(shown in the picture on the right). Then
$\displaystyle{P(T_{2,n})=z^{-1}a^n(a-a^{-1})+z{{1-a^n}\over{a-a^{-1}}}}$.
\end{slds}

\end{minipage}
\begin{minipage}[h]{0.29\linewidth}
\includegraphics[width=\textwidth]{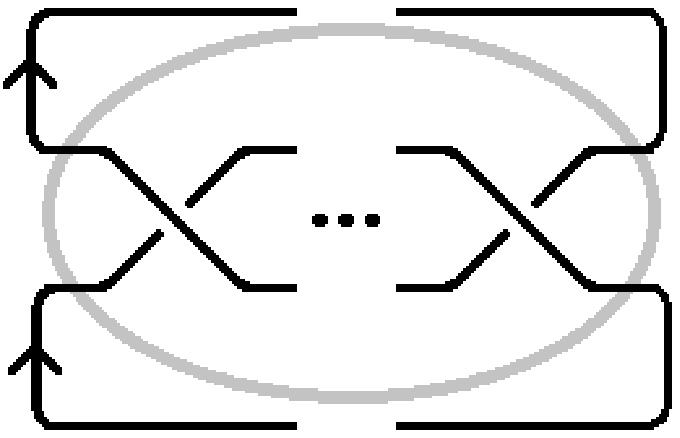}
\end{minipage}

\begin{proof}
Notice that $n$ is even. 
Consider the family of links $L_m=T_{2,m}$, where $m$ is an arbitrary
even number.
Outside of the grey ellipse all the links of this family are the same,
and inside it they look as shown on Fig \ref{even_fragm}.
Therefore, we fall under the assumptions of Lemma \ref{P_of_Ln},
and it only remains to note that
$P(L_0)=z^{-1}(a-a^{-1})$ and $P(L_{\infty})=1$.
\end{proof}

Particular cases of this Corollary for
$n=0,\ \pm 2,\ \pm 4$ give the well-known values of the HOMFLY polynomial
for the two unlinked circles, the Hopf link and the two (out of the total
four) versions of the oriented ``Solomon knot'', see Figure
\ref{homfly_torus}.

\begin{figure}[h]
\begin{center}
\begin{minipage}[h]{0.3\linewidth}
\center{\includegraphics[width=0.5\linewidth]{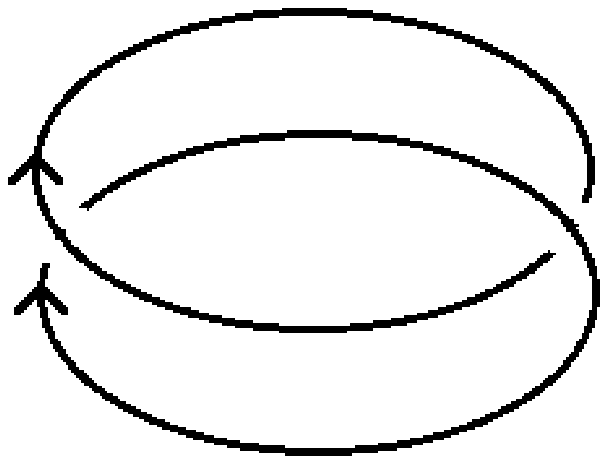}} $z^{-1}(a^3-a)-za$ \\
\end{minipage}\quad
\begin{minipage}[h]{0.25\linewidth}
\center{\includegraphics[width=0.5\linewidth]{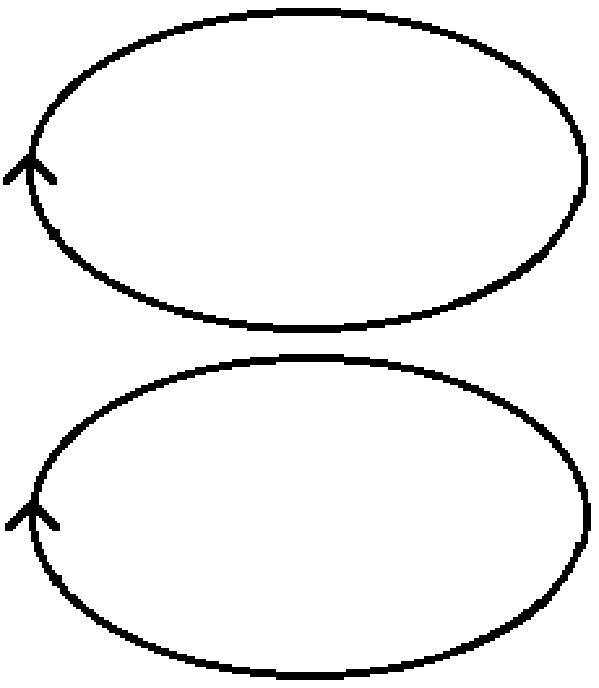}} $z^{-1}(a-a^{-1})$ \\
\end{minipage}\quad
\begin{minipage}[h]{0.3\linewidth}
\center{\includegraphics[width=0.5\linewidth]{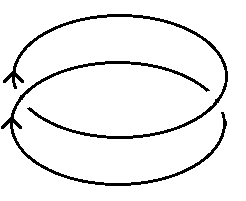}} $z^{-1}(a^{-1}-a^{-3})+za^{-1}$ \\
\end{minipage}

\vspace{3mm}

\begin{minipage}[h]{0.4\linewidth}
\center{\includegraphics[width=0.5\linewidth]{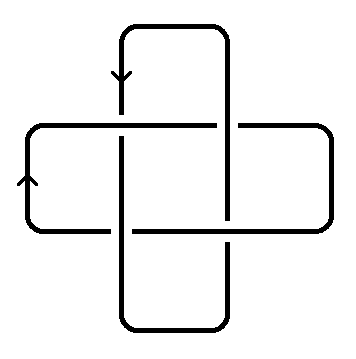}} $z^{-1}(a^5-a^3)-z(a+a^3)$ \\
\end{minipage}
\begin{minipage}[h]{0.4\linewidth}
\center{\includegraphics[width=0.5\linewidth]{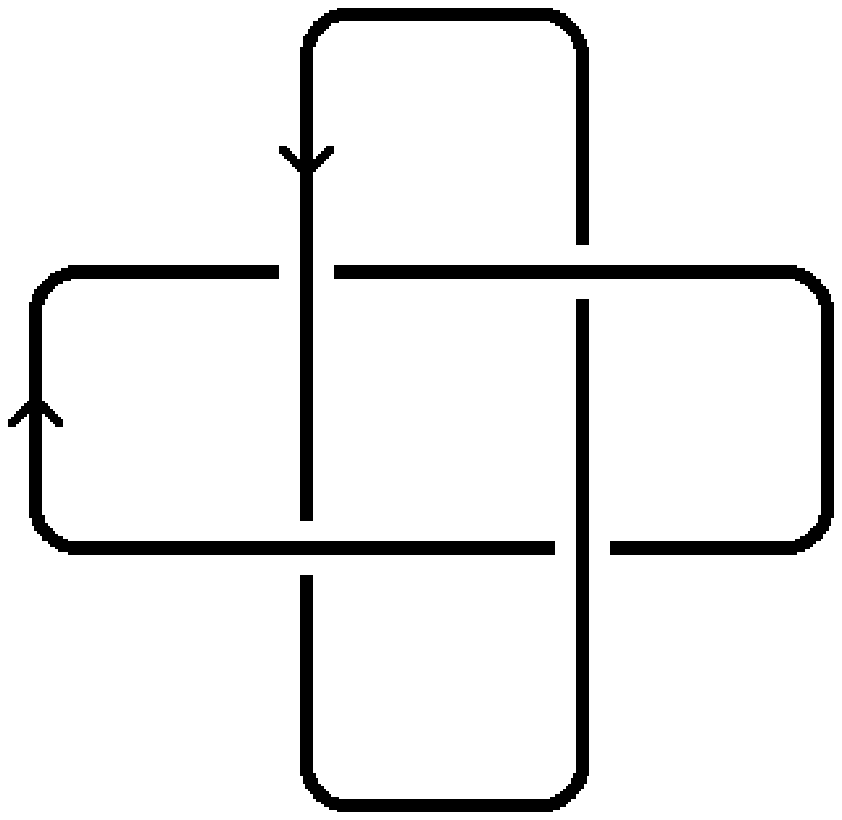}} $z^{-1}(a^{-3}-a^{-5})+z(a^{-1}+a^{-3})$ \\
\end{minipage}
\end{center}
\caption{HOMFLY polynomial of some torus links}
\label{homfly_torus}
\end{figure}

\section{Canonical orientation of rational links}
\label{canon}

\begin{lemma}\label{even_prod}
Suppose the numbers $p$ and $q$ are mutually prime and $|{p\over q}|<1$. 
The number ${p\over q}$ has a continued fraction expansion with non-zero 
even denominators if and only if the product $pq$ is even, and if such
an expansion exists, it is unique.
\end{lemma}

\begin{proof}

1) Necessity: if $p/q=[b_1,b_2,\dots,b_n]$ with all $b_i$'s even, then
$pq$ is even.
We shall prove that by induction on the length $n$ of the
continued fraction. The induction base is evident.
Now,
$$
  {p\over q}=[b_1,b_2,\dots,b_n]={1\over b_1+[b_2,\dots,b_n]}
  ={1\over b_1+p'/q'}={q'\over b_1q'+p'}\ .
$$
By the induction assumption, one (and only one!) of $p'$ or $q'$ is even.
Since $b_1$ is even, it follows that either the numerator or the denominator 
of the last fraction is even, so their product is even and, since the numbers
$p'$ and $q'$ are mutually prime and ${p'}<{q'}$, this fraction is
irreducible and smaller than 1 by absolute value.

2) Sufficiency: if the product $pq$ is even, then the irreducible
fraction $p/q$ allows for a continued fraction with even denominators.

If $q=\pm2$, then the expansion clearly exists.
We proceed by induction on  $|q|$. 
Among the numbers $[{q\over p}]$ and $[{q\over p}]+1$ 
one is even, call it $b$. 
The number $b-{q\over p}$ can be written as an irreducible fraction 
${p'\over q'}$.
Note that $b$ cannot be 0, because $|q/p|>1$.
Then $|{p'\over q'}|<1$ and ${p\over q}={1\over{b+{p'\over q'}}}$, 
where we have $|q'|<|q|$. 
Similarly to the argument in the previous section we infer that
the product $p'q'$ is even. By the induction assumption ${p'\over q'}$ 
has a continued fraction expansion with even denominators.
This completes the proof of sufficiency.

We proceed to the proof of uniqueness, using induction on the length 
of the continued fraction.
For $p=1$ the assertion is trivial. 
Suppose that 
$$
  [b_1,b_2,\dots,b_n]=[c_1,c_2,\dots,c_n]
$$
where all the numbers $b_i$ and $c_i$ are even, and several last terms of
the sequence $c_i$ may be $\infty$ (which means that this sequence is
actually shorter than the first one).
Then
$$
  b_1+[b_2,\dots,b_n]=c_1+[c_2,\dots,c_n].
$$
Therefore,
$$
  |b_1-c_1|=\big|[b_2,\dots,b_n]-[c_2,\dots,c_n]\big| < 2
$$
But the number $|b_1-c_1|$ is even, hence $b_1=c_1$.

The lemma is proved.
\end{proof}

The continued fraction expansion with even denominators and the
corresponding natural diagram will be referred to
as the \textit{canonical expansion} of a rational number
and the \textit{canonical diagram} of a rational link (defined up to a
rotation, see Lemma \ref{dual}).

\begin{zmch*}
The parity of the denominator of a rational number is always opposite to 
the parity of the length of its even (canonical) continued fraction expansion.
That is, for knots the canonical expression is of even length, while for
links it is of odd length.
\end{zmch*}

Now we are in a position to define a canonical \textit{oriented} rational
link.

Let 
$$
{p\over q}=[b_1,b_2,\dots,b_n]
=\cfrac{1}{
b_1+\cfrac{1}{
b_2+\cfrac{1}{
\dots+\cfrac{1}{
b_{n-1}+\cfrac{1}{b_n}}}}},
$$
where $q$ and all $b_i$ are even.  
The diagram $D[b_1,\dots,b_n]$ taken with the positive orientation,
denoted by $D^+[b_1,\dots,b_n]$,
will be referred to as the \textit{canonical diagram} of the oriented link
$L^+(p/q)$.

We will use the canonical diagrams for the proof of the main theorem.
However, for this theorem to make sense, we must check that the oriented
link $L^+(p/q)$ does not depend on a particular choice of the continued
fraction for the rational number $p/q$ and, especially, that it does not
change when $p/q$ is changed by $\bar{p}/q$ where $p\bar{p}\equiv 1 \mod
2q$. We will prove these facts immediately.

\begin{lemma}\label{plusminus}
Suppose that $p/q=[b_1,\dots,b_n]=[c_1,\dots,c_m]$ where $b_i$ and  $c_i$ are
non-zero integers. Then the natural diagram $D^+[b_1,\dots,b_n]$
and $D^+[c_1,\dots,c_m]$  are oriented isotopic.
\end{lemma}
\begin{proof}
We will prove that every natural diagram $D^+[b_1,\dots,b_n]$
is oriented isotopic to the canonical (even) natural diagram.
To do so, we follow the induction argument used in Lemma \ref{even_prod} 
(sufficiency part). 
In fact, the induction step used there consists of one of
the two transformations on the sequence $\{b_1,b_2,\dots,b_n\}$:

(1) $[S,s,t,T]\longmapsto[S,s+1,-1,1-t,T]$, if $t>0$,

(2) $[S,s,t,T]\longmapsto[S,s-1,1,-1-t,T]$, if $t<0$,

\noindent
where $s,t$ are arbitrary integers and $S,T$ are arbitrary sequences.

The algorithm is to find the first from the left occurrence of an odd number
and apply one of these rules.  If $T=\emptyset$ and $t=\pm1$, then we use
the rule $[S,s,\pm1]\mapsto[S,s\pm1]$ instead.  Note that the situation when
all numbers $b_i$, $1\leq i\leq n-1$ are even, while $b_n$ is odd, is
impossible, because it corresponds to a knot rather than to a two-component
link.

The proof of Lemma \ref{even_prod} assures that, in this process, the 
denominator of the rational fraction monotonically decreases, and thus the
algorithm is finite.

Each step of the algorithm, when depicted on natural diagrams, shows that
during this process the equivalence of oriented links is preserved (even with 
numbering of components).
\end{proof}

The previous lemma justifies the notation $L^+(p/q)$.

\begin{lemma}\label{dual}
If $p\bar{p}\equiv 1 \mod 2q$, then the links $L^+(p/q)$ and $L^+(\bar{p}/q)$
are oriented isotopic.
\end{lemma}
\begin{proof}
Making the rotation of the canonical diagram $D^+[b_1,b_2,\dots,b_n]$ around
a vertical axis, we obtain the canonical diagram
$D^+[b_n,b_{n-1},\dots,b_1]$, and it is easy to show (by induction on $n$)
that these two continued fractions have the same denominators, and their
numerators are related as indicated in the statement of the lemma.  (Remind
that, for links, the number $n$ is odd).  We see that the two corresponding 
links are isotopic with the orientation of both components changed.  But the 
total change of orientation is a link equivalence (see \cite{KM}).
\end{proof}

In a canonical diagram of a rational link, due to the fact that all blocks
are of even length, the strands are everywhere counter-directed.
Therefore, Lemma \ref{P_of_Ln} can be applied recursively:
\begin{equation}\label{iterP}
\begin{aligned}
  P(D^+[b_1,\dots,b_n])=&\,a^{\eps_nb_n} P(D^+[b_1,\dots,b_{n-1},0])
  +z{{1-a^{\eps_nb_n}}\over{a-a^{-1}}} P(D^+[b_1,\dots,b_{n-1},\infty])\\
  =&\,a^{\eps_nb_n} P(D^+[b_1,\dots,b_{n-2}])
  +z{{1-a^{\eps_nb_n}}\over{a-a^{-1}}} P(D^+[b_1,\dots,b_{n-1}]),
\end{aligned}
\end{equation}
because the following two pairs of diagrams are equivalent as links:
$D^+[b_1,\dots,b_{n-1},0]=D^+[b_1,\dots,b_{n-2}]$ and 
$D^+[b_1,\dots,b_{n-1},\infty]=D^+[b_1,\dots,b_{n-1}]$. The sign
$\eps_n=(-1)^{n-1}$ comes from our convention of counting the number of twists
in the first and the second layers of a natural diagram (see Figure
\ref{elements} --- the powers of $A$ and $B$ correspond to odd and even 
values of $n$, respectively).

For a given sequence $[b_1,\dots,b_n]$ denote $x_n=P(D^+[b_1,\dots,b_n])$.
Then Equation \ref{iterP} can be rewritten as
\begin{equation}\label{iterPx}
  x_n = z{{1-a^{(-1)^{n-1}b_n}}\over{a-a^{-1}}}x_{n-1} 
      + a^{(-1)^{n-1}b_n}x_{n-2}
\end{equation}
which makes sense when $n>2$. Drawing the diagrams and applying skein
relation (\ref{skein}) for the cases $n=2$ and $n=1$, we can see that 
Equation (\ref{iterPx}) still holds for these values, if
we set $x_0=1$ and $x_{-1}=z^{-1}(a-a^{-1})$.

\section{Main theorem}

Our aim is to find a closed form formula for $x_n$ in terms of $a$ and $z$.
To do this, it is convenient to first consider a more general situation.

\begin{lemma}
\label{recursion}
Let $r_n$ and $l_n$ be elements of a certain commutative ring $R$.
Define recurrently the sequence $x_n$, $n\geq-1$, of polynomials
from $R[z]$ by the relation
$$
  x_n=zl_nx_{n-1} + r_nx_{n-2},\ n\geq 1,
$$
where $x_{-1}$ and $x_{0}$ are fixed elements of $R$. 
Let $C$ be the set of all integer sequences $c=\{c_1,c_2,\dots,c_l\}$
where $c_1>c_2>\dots>c_l$, $c_1=n$, $c_i-c_{i+1}=1\text{ or }2$,
$c_l=0\text{ or }-1$, and only one of the numbers 0 and $-1$ is present in
the sequence $c$ (that is, if $c_l=-1$, then $c_{l-1}\not=0$).
Then $x_n$ can be expressed as the following polynomial in $z$
with coefficients depending on the elements $l_i$, $r_i$ and the initial
conditions $x_0$, $x_{-1}$:
$$
  x_n=\sum_{c\in C} z^{k(c)} x_{c_l}
  \prod_{i\in\l(c)}l_{c_i} \prod_{i\in\r(c)}r_{c_i},
$$
where $\l(c)=\{i\mid c_i-c_{i+1}=1\}$, $\r(c)=\{i\mid c_i-c_{i+1}=2\}$
and $k(c)=|\l(c)|=\#\{i\mid c_i-c_{i+1}=1\}$.
\end{lemma}

\begin{proof}\strut

\begin{minipage}[h]{0.51\linewidth} 
Essentially, the written formula describes the computational tree for
the calculation of $x_n$. 
Note that the recurrence is of depth 2, that is, the element $x_n$
is expressed through $x_{n-1}$ and $x_{n-2}$. Therefore, the computational
tree is best represented as a layered tree where each layer matches the
$l_i$'s and $r_i$'s with the same $i$. We draw the $l$-edges (of length 1)
to the left and the $r$-edges (of length 2) to the right.
The exponent of $z$ for each directed path from the vertex
at level $n$ to a vertex at levels 0 or $-1$ in this tree
corresponds to the number of left-hand edges. Any path in such a tree is
uniquely determined by a sequence of levels $c$ with the above listed 
properties.
In the picture, you can see an example of such a tree for $n=5$.
\end{minipage}
\quad
\begin{minipage}[h]{0.39\linewidth}
\begin{flushright}
\includegraphics[width=\textwidth]{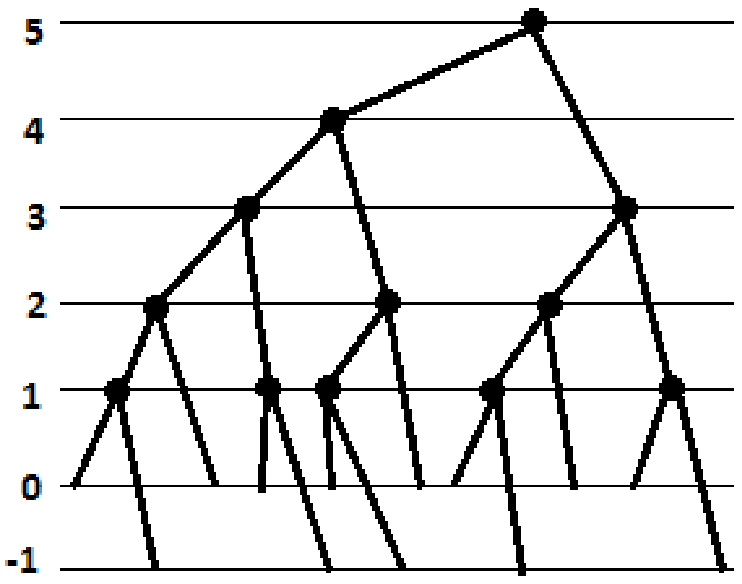}
\end{flushright}
\end{minipage} 
\smallskip

\end{proof}

To obtain the formula for the HOMFLY polynomial of an arbitrary rational
link $L^\pm(p/q)$, we combine Lemma \ref{recursion} with formula \ref{iterPx}.
For the sake of unification, we first make some preparations:
\begin{itemize}
\item
If $q$ is odd (that is, we deal with a knot) and $p$ is odd, too,
then we change $p$ to $p'=p-q$, if $p>0$, or to $p'=p+q$, if $p<0$.
Then $L(p/q)=L(p'/q)$, the numerator of the fraction becomes even, 
hence Lemma \ref{even_prod} applies and formula (\ref{iterPx}) is valid.
\item
If $q$ is even and the link is negative, then we use the property
$L^-(p/q)=L^+(p'/q)$, where $p'$ is computed by the same rule as above.
Below, we will simply write $L(p/q)$ instead of $L^+(p/q)$.
\end{itemize}

Now the main result reads:

\begin{theorem}
Suppose that $p$ is even and $q$ is odd or $p$ is odd and $q$ is even.
Let $[b_1,b_2,\dots,b_n]$ be the canonical continued fraction 
for the number $p/q$ (all numbers $b_i$ are even, positive or negative,
see Lemma \ref{even_prod}). Then 
\begin{equation}\label{main}
P(L(p/q))=
  \sum_{c\in C} z^{k(c)} x_{c_l}
  \prod_{i\in\l(c)}{1-a^{(-1)^{c_i-1}b_{c_i}}\over a-a^{-1}} 
  \prod_{i\in\r(c)}a^{(-1)^{c_i-1}b_{c_i}},
\end{equation}
where
\begin{itemize}
\item
$C$ is the set of all integer sequences $c=\{c_1,c_2,\dots,c_l\}$
with $c_1>c_2>\dots>c_l$, $c_i-c_{i+1}=1\text{ or }2$, $c_1=n$,
$c_l=0\text{ or }-1$, and only one of the numbers $0$ and $-1$ is present in
the sequence $c$ (that is, if $c_l=-1$, then $c_{l-1}\not=0$),
\item
$\l(c)=\{i\mid c_i-c_{i+1}=1\}$,
\item
$\r(c)=\{i\mid c_i-c_{i+1}=2\}$,
\item
$k(c)=|\l(c)|=\#\{i\mid c_i-c_{i+1}=1\}$,
\item
$x_0=1$ and $x_{-1}=z^{-1}(a-a^{-1})$.
\end{itemize}
\end{theorem}

\begin{proof}
The proof was actually given above.
\end{proof}

\textbf{Example.}
The canonical expansion of the fraction $4/7$ is $[2,-4]$. We have $n=2$,
and there are three possibilities for the sequence $c$:

(1) $c=\{2,1,0\}$, then $\l(c)=\{2,1\}$, $k(c)=2$,
$\r(c)=\emptyset$, $c_l=0$,

(2) $c=\{2,1,-1\}$, then $\l(c)=\{2\}$, $k(c)=1$,
$\r(c)=\{1\}$, $c_l=-1$,

(2) $c=\{2,0\}$, then $\l(c)=\emptyset$, $k(c)=0$,
$\r(c)=\{2\}$, $c_l=0$,

\noindent
Then formula (\ref{main}) gives:

\begin{align*}
  P(L(4/7))&=z^2\cdot{1-a^2\over a-a^{-1}}\cdot{1-a^4\over a-a^{-1}}
          +z\cdot z^{-1}(a-a^{-1})\cdot{1-a^4\over a-a^{-1}}\cdot a^2 +a^4\\
          &=z^2(a^2+a^4)+(a^2+a^4-a^6)
\end{align*}

In formula \ref{main} one can, in principle, collect the terms with equal
powers of $z$. The formulation of this result is rather involved, and we
need first to introduce necessary notations.

Let $\alpha=p/q$ be a nonzero irreducible rational number between $-1$ and $1$.
We denote by $n=\nu(\alpha)$ the length of the canonical continuous fraction
for $\alpha$, and by $\alpha'$, the number $\alpha+1$, if $\alpha<0$,
and $\alpha-1$, if $\alpha>0$. Now, let
$$
\rho_k(\alpha)=
\sum_{\substack{{C\subseteq{{\overline{1,n}}},\ C\cap(C-1)=\O\ } \\ 
{|C|=(n-k)/2}}}\prod_{m\in C}a^{(-1)^{m+1}b_m}
\prod_{\substack{{m\in{\overline{1,n}}} 
\\ {m\notin C\cup(C-1)}}}(1-a^{(-1)^{m+1}b_m})
$$
where $\overline{1,n}=\{1,2,\dots,n\}$ and $C-1$ is understood as the set 
of all numbers $c-1$, where $c\in C$.

Then we have:

\begin{theorem}
Let $q$ be odd, that is, $L(\alpha)$ is a knot. Then:

1) If $p$ is even, then 
$$
P(L(\alpha))=\sum_{\substack{{{0\leq k\leq \nu(\alpha)}} \\
{k\equiv 0 \mod 2}}}z^{k}(a-a^{-1})^{-k}\rho_{k}(\alpha)
$$

2) If $p$ is odd, then 
$$
P(L(\alpha))=\sum_{\substack{{{0\leq k\leq \nu(\alpha')}} \\
{k\equiv 0 \mod 2}}}z^{k}(a-a^{-1})^{-k}\rho_{k}(\alpha')
$$

Let $q$ be even, that is, $L(\alpha)$ is a two-component link.
Then:

3) If the two components are counterdirected, then
$$
P(L^{+}(\alpha))=\sum_{\substack{{{-1\leq k\leq \nu(\alpha)}} \\ 
{k\equiv 1 \mod 2}}}z^{k}(a-a^{-1})^{-k}\rho_{k}(\alpha)
$$

4) If the two components are codirected, then
$$
P(L^{-}(\alpha))=\sum_{\substack{{{-1\leq k\leq \nu({\alpha'})}} \\ 
{k\equiv 1 \mod 2}}}z^{k}(a-a^{-1})^{-k}\rho_{k}({\alpha'})
$$
\end{theorem} 

The theorem can be proved by first collecting the terms with equal powers
of $z$ in the statement of Lemma \ref{recursion} and then using
induction on $\nu(\alpha)$; we do not give the details here.
Although Theorem 2 is in a sense more explicit than Theorem 1,
it is less practical; in particular, the formula of
Theorem 1 is better suited for programming purposes.

\section{Computer calculations}

The formula for $P(L(p/q))$ can be easily programmed.
The source code of the program, written by the second author and tested by
the first one, as well as the resulting table of HOMFLY polynomials for
rational links with denominators not exceeding 1000, are presented online
at \cite{Calc}. Below, we give a short excerpt of that big table which is
enough to know the polynomials of all rational knots and links with 
denominators no greater than 9, if one uses the following rules 
(see \cite{KM}):
\begin{enumerate}
\item
$P(L^+(-p/q))$ is obtained from $P(L^+(p/q))$ by the substitution 
$a\mapsto -a^{-1}$.
\item
The knots $L(p_1/q)$ and $L(p_2/q)$ are equivalent, if 
$p_1p_2\equiv 1\mod q$.
\item
The links $L^+(p_1/q)$ and $L^+(p_2/q)$ are oriented equivalent, if 
$p_1p_2\equiv 1\mod 2q$.
\end{enumerate}

In Table \ref{homflytab}, the first column (R) gives the notation
of the rational link (knot) as $L(p/q)$ (in the case of links, this means
$L^+(p/q)$), the second column (T) contains the standard notation of that
link (knot) from Thistlethwaite (Rolfsen) tables (see \cite{Atlas};
the bar over a symbol means mirror reflection, the star is for the change
of orientation of one component),
and the third column (H) is for the values of the HOMFLY polynomial.
Note that we list HOMFLY polynomials for both orientations of each
rational link, while the famous Knot Atlas \cite{Atlas} shows them for 
only one orientation of two-component links.

\begin{table}[htb]
$$\begin{array}{c|c|l}
R      &   T   &      H      \\
\hline
L({1\over2}) & L_2a_1 &  z^{-1}(a^3-a)-za\rule{0pt}{15pt} \\[1mm]
L({2\over3}) &  3_1  &  (2a^2-a^4)+z^2a^2 \\[1mm]
L({1\over4}) & L_4a_1 & z^{-1}(-a^3+a^5) + z(-a-a^3)\\[1mm]
L({3\over4}) & L_4a_1^* & z^{-1}(-a^3+a^5)+z(-3a^3+a^5)-z^3a^3\\[1mm]
L({2\over5}) &  4_1 & (a^{-2}-1+a^2)-z^2 \\[1mm]
L({4\over5}) & 5_1 & (3a^4-2a^6)+z^2(4a^4-a^6)+z^4a^4 \\[1mm]
L({1\over6}) & L_6a_3 & z^{-1}(-a^5+a^7)+z(-a-a^3-a^5) \\[1mm]
L({5\over6}) & L_6a_3^* & z^{-1}(a^7-a^5)+z(3a^7-6a^5)+z^3(a^7-5a^5)-z^5a^5\\[1mm]
L({2\over7}) & 5_2 & (a^2+a^4-a^6)+z^2(a^2+a^4) \\[1mm]
L({6\over7}) & 7_1 & (4a^6-3a^8)+z^2(10a^6-4a^8)+z^4(6a^6-a^8)+z^6a^6\\[1mm]
L({1\over8}) & L_8a_{14}^* & z^{-1}(-a^7+a^9)+z(-a-a^3-a^5-a^7)\\[1mm]
L({3\over8}) & \overline{L_5a_1} & z^{-1}(-a^{-1}+a)+z(a^{-3}-2a^{-1}+a)-z^3a^{-1}\\[1mm]
L({7\over8}) & L_8a_{14} & z^{-1}(a^9-a^7)+z(6a^9-10a^7)+z^3(5a^9-15a^7)
               +z^5(a^9-7a^7)-z^7a^7\\[1mm]
L({2\over9}) & 6_1 & (a^{-2}-a^2+a^4)+z^2(-1-a^2)\\[1mm]
L({8\over9}) & 9_1 & (5a^8-4a^{10})+z^2(20a^8-10a^{10})+z^4(21a^8-6a^{10})
               +z^6(8a^8-a^{10})+z^8a^8
\end{array}$$
\caption{HOMFLY polynomials of rational links with denominators $\le9$}
\label{homflytab}
\end{table}

\section{Concluding remarks}

\indent

1. As the Conway polynomial is a reduction of the HOMFLY polynomial, 
Theorem 1 gives a formula for the Conway polynomial of rational links
by the substitutions $a=1$, $z=t$ (the fraction $(1-a^n)/(a-a^{-1})$
is first transformed to a Laurent polynomial and becomes equal to 
$-n/2$).

2. Since the Jones polynomial is a reduction of the HOMFLY polynomial, 
Theorem 1 leads to a formula for the Jones polynomial of rational links
by the substitutions $a=t^{-1}$, $z=t^{1/2}-t^{-1/2}$.

3. The famous open problem whether a knot must be trivial if its Jones 
polynomial is 1, has a simple positive solution for rational knots.
Indeed, the value $|J(-1)|$ is equal to the determinant of the knot,
and the determinant of a rational knot is its denominator (see \cite{KM}).

4. \textit{Open problem}. Can one generalize formula \ref{main}
to \textit{all} links? The results of L.\,Traldi \cite{Tr} show that
it can be generalized to at least some non-rational links, although his
formula is less explicit than ours.

\section{Acknowledgements}

The authors are grateful to M.\,Karev who read the manuscript and indicated
several inaccuracies. We also thank S.\,Chmutov for pointing out the
relation of our investigations with papers \cite{Jaeg} and \cite{Tr},
and L.\,Traldi for valuable comments on his paper \cite{Tr}.

\end{document}